%% file: main.tex
\documentclass[12pt]{amsart}

\usepackage[margin = 1in]{geometry}
\usepackage{amsfonts, amsmath, amscd, amssymb,
latexsym, mathrsfs, mathtools, 
slashed, stmaryrd, verbatim, wasysym, bbm, url}
\usepackage[utf8]{inputenc}
\usepackage[backend=bibtex,bibstyle=authoryear,citestyle=alphabetic,giveninits=true,doi=false,isbn=false,url=false,backref]{biblatex}
\usepackage[usenames]{xcolor}
\usepackage{enumerate}
\usepackage[shortlabels]{enumitem}
\usepackage{scalerel}
\usepackage{soul}
      
\DeclareFieldFormat{labelalphawidth}{\mkbibbrackets{#1}}

\defbibenvironment{bibliography}
  {\list
     {\printtext[labelalphawidth]{%
        \printfield{labelprefix}%
    \printfield{labelalpha}%
        \printfield{extraalpha}}}
     {\setlength{\labelwidth}{\labelalphawidth}%
      \setlength{\leftmargin}{\labelwidth}%
      \setlength{\labelsep}{\biblabelsep}%
      \addtolength{\leftmargin}{\labelsep}%
      \setlength{\itemsep}{\bibitemsep}%
      \setlength{\parsep}{\bibparsep}}%
      }
  {\endlist}
  {\item}
\DeclareNameAlias{author}{last-first}

\newtheorem*{acknowledgements}{Acknowledgements}
\newtheorem*{theorem*}{Theorem}
\newtheorem{theorem}{Theorem}
\newtheorem{corollary}{Corollary}

\newtheorem{proposition}[corollary]{Proposition}
\theoremstyle{definition}
\newtheorem{definition}{Definition}
\newtheorem{example}{Example}
\newtheorem{remark}[example]{Remark}

\numberwithin{equation}{section}
\let\oldsqrt\sqrt
\def\sqrt{\mathpalette\DHLhksqrt}
\def\DHLhksqrt#1#2{%
\setbox0=\hbox{$#1\oldsqrt{#2\,}$}\dimen0=\ht0
\advance\dimen0-0.2\ht0
\setbox2=\hbox{\vrule height\ht0 depth -\dimen0}%
{\box0\lower0.4pt\box2}}

\usepackage{verbatim} %% code from mathabx.sty and mathabx.dcl
\DeclareFontFamily{U}{mathx}{\hyphenchar\font45}
\DeclareFontShape{U}{mathx}{m}{n}{
      <5> <6> <7> <8> <9> <10>
      <10.95> <12> <14.4> <17.28> <20.74> <24.88>
      mathx10
      }{}
\DeclareSymbolFont{mathx}{U}{mathx}{m}{n}
\DeclareFontSubstitution{U}{mathx}{m}{n}
\DeclareMathAccent{\widecheck}{0}{mathx}{"71}

\newcommand\eps\varepsilon
\renewcommand\epsilon\varepsilon

\newcommand{\abs}[1]{\left\lvert #1 \right\rvert}
\newcommand{\smallabs}[1]{\lvert #1 \rvert}

\newcommand\floor[1]{\lfloor #1 \rfloor}

\newcommand{\norm}[1]{\lVert #1 \rVert}
\newcommand\inner[1]{\langle #1 \rangle}

\newcommand\sign{\operatorname{sign}}

\newcommand\Mand{\text{ and }}

\newcommand\paperintro%
        {%
         }
\newcommand\paperbody%
        {%
         }

\newcommand\bbG{\mathbb{G}}
\newcommand\bbH{\mathbb{H}}

\newcommand\bbR{\mathbb{R}}
\newcommand\bbS{\mathbb{S}}

\newcommand\bbZ{\mathbb{Z}}

\newcommand\cH{\mathcal{H}}
\newcommand\cI{\mathcal{I}}
\newcommand\cJ{\mathcal{J}}

\DeclareMathAlphabet{\mathpzc}{OT1}{pzc}{m}{it}

\newcommand{\sbs}{\subset}

\usepackage[refpage]{nomencl}

\makenomenclature

\usepackage{tikz, tikz-cd, tkz-euclide}
\usetikzlibrary{calc,positioning,intersections,quotes,decorations.markings}
\makeatletter
\def\@tocline#1#2#3#4#5#6#7{\relax
  \ifnum #1>\c@tocdepth % then omitth
  \else
    \par \addpenalty\@secpenalty\addvspace{#2}%
    \begingroup \hyphenpenalty\@M
    \@ifempty{#4}{%
      \@tempdima\csname r@tocindent\number#1\endcsname\relax
    }{%
      \@tempdima#4\relax
    }%
    \parindent\z@ \leftskip#3\relax \advance\leftskip\@tempdima\relax
    \rightskip\@pnumwidth plus4em \parfillskip-\@pnumwidth
    #5\leavevmode\hskip-\@tempdima
      \ifcase #1
       \or\or \hskip 1em \or \hskip 2em \else \hskip 3em \fi%
      #6\nobreak\relax
    \hfill\hbox to\@pnumwidth{\@tocpagenum{#7}}\par% <---- \dotfill -> \hfill
    \nobreak
    \endgroup
  \fi}
\makeatother

\def\annu#1{_{% 
  \vbox{\hrule height .2pt 
    \kern 1pt 
    \hbox{$\scriptstyle {#1}\kern 1pt$}% 
  }\kern-.05pt 
  \vrule width .2pt 
}}

\allowdisplaybreaks
\usepackage[colorlinks,citecolor=blue,urlcolor=red,bookmarks=false,hypertexnames=true]{hyperref} 

\makeatletter
\def\keywords{\xdef\@thefnmark{}\@footnotetext}
\makeatother

\title[Almost isoperimetric extremisers of two subriemannian probability measures]{Almost isoperimetric extremisers of two subriemannian probability measures}
\author{Yaozhong Qiu}
\address{UPL, Univ. Paris Nanterre, CNRS, F92000 Nanterre France}
\email{yqiu@parisnanterre.fr}

\begin{document}
\begin{abstract}
We prove the existence of almost isoperimetric extremisers for two classes of probability measures defined respectively on the Grushin space and a stratified Lie group. It turns out such extremisers can be regarded as a type of anisotropic half-space.  
\end{abstract}

\keywords{2020 \emph{Mathematics Subject Classification.} Primary 53C17, 49Q20}
\keywords{\emph{Keywords.} Isoperimetric inequality, extremisers, Grushin space, stratified Lie group, probability measure}

\maketitle

\input{introduction}
\input{grushin}
\input{heisenberg}

\begin{acknowledgements}
\textup{We are grateful to Daniel R. Johnston for helpful discussions. We would also like to thank Nazarbayev University for their support and hospitality during the 15th ISAAC Congress where part of the work on this paper was undertaken. This project has received funding from the European Union’s Horizon 2020 research and innovation programme under the Marie Skłodowska-Curie grant agreement No 101034255. \scalerel*{\includegraphics{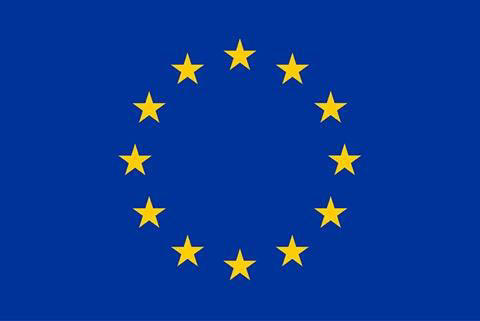}}{A}}
\end{acknowledgements}

\printbibliography

\end{document}

%% file: introduction.tex
\section{Introduction and main results}
In this paper, we prove the existence of almost isoperimetric extremisers for two classes of probability measures defined respectively on the Grushin space and a stratified Lie group. In doing so, we continue the study of the isoperimetric problem for such measures initiated by the author in \cite{qiu2024optimal}. By \emph{almost}, we mean the following. 

\begin{definition}
    Let $(X, \mu, d)$ be a metric probability measure space. The perimeter of a Borel set $A \sbs X$ is the lower Minkowski content
    \begin{equation}\label{def-perimeter}
        \mu^+(A) = \liminf_{\epsilon \to 0^+} \frac{\mu(A_\epsilon) - \mu(A)}{\epsilon}
    \end{equation}
    where $A_\epsilon = \{x \in X \mid d(x, A) < \epsilon\}$, and the isoperimetric profile of $\mu$ is 
    \begin{equation}\label{def-isoperimetric-profile}
        \cI_\mu(t) = \inf\{\mu^+(A) \mid A \text{ is Borel and } \mu(A) = t \in (0, 1)\}, 
    \end{equation}
    so that $\cI_\mu$ is the optimal function satisfying the inequality 
    \[ \mu^+(A) \geq \cI_\mu(\mu(A)). \] 
    A family of sets $(A_t)$ of sufficiently small measure $0 < \mu(A_t) = t \leq t_0 < 1$ having the property there exists a constant $C > 1$ such that 
    \begin{equation}\label{def-almost-isoperimetric}
        C \geq \frac{\mu^+(A_t)}{\cI_\mu(\mu(A_t))} 
    \end{equation}
    for all $0 < t \leq t_0 < 1$ is called a \emph{family of almost isoperimetric extremisers}. 
\end{definition}

We will restrict our attention to the following two settings. We first consider the simpler case of the Grushin space $\bbR^{n+m} = \bbR^n_x \times \bbR^m_y$ equipped, for $\gamma \geq 0$, with its subgradient $\nabla_\gamma = (\nabla_x, \abs{x}^\gamma \nabla_y)$ and sublaplacian $\Delta_\gamma = \nabla_\gamma \cdot \nabla_\gamma$. The fundamental solution of $\Delta_\gamma$ is given by a power of  
\[ N_\gamma(x, y) = (\abs{x}^{2(1+\gamma)} + (1 + \gamma)^2\abs{y}^2)^{1/(2(1+\gamma)} \]
up to constants. It was shown in \cite[Corollary~6]{qiu2024optimal} that the probability measure 
\begin{equation}\label{def-grushin-measure}
    d\mu_{\gamma, \, p}(x, y) = Z^{-1}\exp(-N_\gamma^p)dxdy, 
\end{equation}
for $p \geq 1 + \gamma$ and $Z = Z_{\gamma, \, p}$ a normalisation constant, satisfies the isoperimetric inequality 
\begin{equation}\label{isoperimetric-grushin}
    \cI_{\mu_{\gamma, \, p}} \gtrsim \cJ_{p(1+\gamma)/(p\gamma + 1 + \gamma)}
\end{equation}
with respect to the Carnot-Carath\'eodory metric $d_\gamma$ induced by $\nabla_\gamma$ and where $\cJ_r$ is the isoperimetric profile of the measure $d\nu_r = Z^{-1}\exp(-\abs{x}^r)dx$ on $\bbR$ with respect to the euclidean metric which, according to \cite[Proposition~13.4]{bobkov1997some}, behaves asymptotically like 
\begin{equation}\label{isoperimetric-model-asymptotics}
    \cJ_r(t) = \cI_{\nu_r}(t) \simeq t \log(1/t)^{1-1/r}.
\end{equation}
It was established in \cite{qiu2024optimal} that \eqref{isoperimetric-grushin} is optimal for $\gamma \in \bbZ_{\geq 1}$ in the sense $p(1 + \gamma)/(p\gamma + 1 + \gamma)$ cannot be improved. For instance, if $\gamma = 1$ and $p = 4$, then $\mu_{\gamma, \, p} = \mu_{1, \, 4}$ has the supergaussian tails of $\nu_4$ with respect to $d_\gamma = d_1$ but the subgaussian isoperimetric profile of $\nu_{4/3}$. Moreover, since $p(1 + \gamma)/(p\gamma + 1 + \gamma) = 2p/(p + 2) < 2$ for each $p \geq 2$, we see $\mu_{1, \, p}$ never achieves the gaussian isoperimetric inequality.

This paper is motivated by the question of whether one can find a family of almost isoperimetric extremisers $(A_t)$ for $\mu_{\gamma, \, p}$ which would not only confirm the optimality of \eqref{isoperimetric-grushin}, but also shed some light on the geometry of the actual extremisers. It turns out that almost extremisers are a type of anisotropic half-space and so \eqref{isoperimetric-grushin} can be regarded as a generalisation of the gaussian isoperimetric inequality modulo constants. 

\begin{theorem}\label{thm-grushin}
    For the metric probability measure $(\bbR^n_x \times \bbR^m_y, \mu_{\gamma, \, p}, d_\gamma)$, $\gamma \geq 0$, and $p \geq 1 + \gamma$, sets of the form
    \[ A_K = \{(x, y) \in \bbR^n_x \times \bbR^m_y \mid (1 + \gamma)\abs{y} \geq \abs{x}^{1+\gamma} + K\} \]
    form a family of almost isoperimetric extremisers in the sense of \eqref{def-almost-isoperimetric} as $K \rightarrow \infty$. 
\end{theorem}

The proof of this result is via explicit computation and relies on the fact the perimeter $\mu_{\gamma, \, p}^+(A)$ of sets $A$ with Lipschitz boundary enjoy the representation formula
\begin{equation}\label{perimeter-grushin}
    \mu^+_{\gamma, p}(A) = \int_{\partial A} \sqrt{\abs{N_x}^2 + \abs{x}^{2\gamma}\abs{N_y}^2} \varphi_{\gamma, p}(x, y) d\cH^{n+m-1}(x, y)
\end{equation}
where $(N_x, N_y)$ is the unit euclidean normal, $\varphi_{\gamma, p}$ is the density of $\mu_{\gamma, \, p}$ with respect to Lebesgue measure, and $\cH^{n+m-1}$ is the $(n+m-1)$-dimensional Hausdorff measure, see for instance \cite[Proposition~3.1]{monti2014isoperimetric}.

We then consider the $3$-dimensional Heisenberg group $\bbH^1$ and prove the analogous result but with a different method since a formula for the perimeter is not known to us. It turns out that the proof extends to the setting of a step two stratified Lie group $\bbG$ which is defined here as $\bbR^n_x \times \bbR^m_z$ equipped with a group law of the form 
\begin{equation}\label{def-step-two-law} 
    (x, z) \circ (\xi, \zeta) = \left(x + \xi, z_1 + \zeta_1  + \frac{1}{2}\inner{B^{(1)}x, \xi}, \cdots, z_m + \zeta_m + \frac{1}{2}\inner{B^{(m)}x, \xi}\right)
\end{equation}
for a collection of $m$ linearly independent skew-symmetric matrices $B^{(1)}, \cdots, B^{(m)}$ of dimension $n \times n$. The group law gives rise to a family of canonical vector fields $X_1, \cdots, X_n$ which form the subgradient $\nabla_\bbG$ and the sublaplacian $\Delta_\bbG = \nabla_\bbG \cdot \nabla_\bbG$. Although the analogue 
\begin{equation}\label{def-kaplan-norm}
    N_\bbG(x, z) = (\abs{x}^4 + \abs{z}^2)^{1/4}
\end{equation}
of $N_\gamma$ is not in general the fundamental solution of $\Delta_\bbG$, except on a special class of groups called the $H$-type groups, see \cite[\S3.6~and~\S18]{bonfiglioli2007stratified}, this function is a homogeneous norm in the sense of \cite[\S5.1]{bonfiglioli2007stratified} which we call the Kaplan norm after \cite{kaplan1980fundamental}. It was shown in \cite[Corollary~5]{qiu2024optimal} that the analogous probability measure 
\begin{equation}\label{def-group-measure}
    d\mu_{\bbG, \, p} = Z^{-1}\exp(-N_{\bbG}^p)dxdz
\end{equation}  
for $p \geq 2$ and $Z = Z_p$ a normalisation constant, satisfies the isoperimetric inequality 
\begin{equation}\label{isoperimetric-group}
    \cI_{\mu_{\bbG, \, p}} \gtrsim \cJ_{2p/(p+2)},
\end{equation}
that is \eqref{isoperimetric-grushin} at $\gamma = 1$, and it turns out that the almost extremisers are exactly as before. 
\begin{theorem}\label{thm-group}
    For the metric probability measure $(\bbG \cong \bbR^n_x \times \bbR^m_z, \mu_{\bbG, \, p}, d_{\bbG})$, sets of the form
    \[ A_K = \{(x, z) \in \bbG \mid \abs{z} > \abs{x}^2 + K\} \]
    form a family of almost isoperimetric extremisers in the sense of \eqref{def-almost-isoperimetric} as $K \rightarrow \infty$. 
\end{theorem}

To conclude this introduction, we remark in the simple setting of the Grushin plane where $n = m = 1$ and assuming $\gamma = 1$ that these almost extremisers can be taken as sets of the form $\{y > x^2 + K\}$ in contrast with the half-spaces $\{y > x + K\}$ extremising the gaussian isoperimetric inequality. We hope this provides new geometric insight into these probability measures, for instance with regards to the possibility of defining either an analogue of L\'evy's spherical isoperimetric inequality, in light of its connection with the gaussian isoperimetric inequality due to Borell, Sudakov, and Tsirel'son \cite{borell1975brunn, sudakov1978extremal}, or otherwise an analogue of the gaussian rearrangement developed by Ehrhard and Borell \cite{ehrhard1983symetrisation, borell2003ehrhard}. Finally, we refer the reader to \cite{ledoux2006isoperimetry} for a modern account of the gaussian isoperimetric inequality, \cite{bonfiglioli2007stratified} for a monograph on stratified Lie groups, and \cite{monti2004isoperimetric, monti2014isoperimetric, franceschi2016isoperimetric} for other related isoperimetric inequalities and problems. 

%% file: grushin.tex
\section{The Grushin setting} 
Since we have already given the sets, all that remains is to compute their volume and perimeter. Let us first note 
\begin{equation}\label{exponential-asymptotic}
    \int_x^\infty e^{-\phi(y)}dy = \int_x^\infty \frac{d}{dy}\left(-\frac{1}{\varphi'(y)}e^{-\phi(y)}\right)dy\simeq \frac{e^{-\phi(x)}}{\phi'(x)} 
\end{equation}
for large $x$ and $\abs{x}^p$ for some $p > 0$ at infinity. Although we need only an upper bound on $\mu_{\gamma, \, p}(A_K)$ because $t\log(1/t)^{1-1/r}$ is increasing near $t = 0$ for each $r \geq 1$, we obtain two-sided asymptotics for completeness. To simplify notation in the sequel we write $\alpha = 1 + \gamma$ and $p = 2\beta(1 + \gamma) = 2\alpha\beta$ with $\beta \geq \frac{1}{2}$. 

\subsection{Estimates for the volume}
For the volume, passing to radial coordinates and changing variables, we have 
\begin{align}
    \mu_{\gamma, \, p}(A_K) &= Z^{-1} \int_{A_K} \exp(-\abs{x}^{2(1+\gamma)} + (1+\gamma)^2\abs{y}^2)^{p/(2(1+\gamma))}dydx \notag \\ 
    &\simeq \int_0^\infty \int_{\alpha^{-1}(x^\alpha + K)}^\infty x^{n-1}y^{m-1}e^{-(x^{2\alpha} + \alpha^2y^2)^\beta}dydx \notag \\
    &\simeq \int_0^\infty \int_{x^\alpha + K}^\infty x^{n-1}y^{m-1}e^{-(x^{2\alpha} + y^2)^\beta}dydx. \label{volume-start}
\end{align}
For the upper bound, we have $(x^{2\alpha} + y^2)^\beta \geq y^{2\beta}$ and therefore by \eqref{exponential-asymptotic} 
\begin{align}
    \mu_{\gamma, \, p}(A_K) &\lesssim \int_0^\infty \int_{x^\alpha + K}^\infty x^{n-1}y^{m-1}e^{-y^{2\beta}}dydx \vphantom{\int_0^\infty x^{n-1}\frac{(x^\alpha + K)^{m-1}e^{-(x^\alpha + K)^{2\beta}}}{(x^\alpha + K)^{2\beta-1} - (m-1)\log(x^\alpha + K)}dx} \notag \\
    &\simeq \int_0^\infty x^{n-1}\frac{(x^\alpha + K)^{m-1}e^{-(x^\alpha + K)^{2\beta}}}{(x^\alpha + K)^{2\beta-1} - (m-1)\log(x^\alpha + K)}dx \notag \\
    &\lesssim \int_0^\infty x^{n-1}(x^\alpha + K)^{m-2\beta} e^{-(x^\alpha + K)^{2\beta}}dx. \vphantom{\int_0^\infty x^{n-1}\frac{(x^\alpha + K)^{m-1}e^{-(x^\alpha + K)^{2\beta}}}{(x^\alpha + K)^{2\beta-1} - (m-1)\log(x^\alpha + K)}dx} \label{volume-intermediate-1}
\end{align}
To elucidate the argument, we factor out $K$ from both parentheses before making the first change of variable $z = K^{-1}x^\alpha$ to find
\begin{align*}
    \mu_{\gamma, \, p}(A_K) &\lesssim K^{m-2\beta} \int_0^\infty x^{n-1}\left(K^{-1}x^\alpha + 1\right)^{m-2\beta}e^{-K^{2\beta}(K^{-1}x^\alpha + 1)^{2\beta}}dx \\
    &= K^{m-2\beta} K^{1/\alpha} \int_0^\infty K^{(n-1)/\alpha} z^{n/\alpha-1}(z + 1)^{m-2\beta} e^{-K^{2\beta}(z + 1)^{2\beta}}dz
\end{align*}
and after the second change of variable $w = K^{2\beta}z$ we obtain
\begin{align*}
    \mu_{\gamma, \, p}(A_K) &\lesssim K^{m-2\beta+n/\alpha}K^{-2\beta} \int_0^\infty K^{-2\beta(n/\alpha-1)}w^{n/\alpha-1}(K^{-2\beta}w + 1)^{m-2\beta}e^{-K^{2\beta}(K^{-2\beta}w + 1)^{2\beta}}dw \\
    &= K^{m-2\beta-(2\beta-1)n/\alpha}\int_0^\infty w^{n/\alpha-1}(K^{-2\beta}w + 1)^{m-2\beta}e^{-K^{2\beta}(K^{-2\beta}w + 1)^{2\beta}dw}. 
\end{align*}
If $2\beta \geq 1$ were an integer then we could just extract the factor $\exp(-K^{2\beta})$ from the exponential and what remains is bounded above by a finite integral. To generalise to $2\beta$ noninteger, we apply Bernoulli's inequality to obtain $(1 + K^{-2\beta}w)^{2\beta} \geq 1 + 2\beta K^{-2\beta}w$ and argue similarly. Either way, we conclude the volume of $A_K$ enjoys the asymptotics
\begin{equation}\label{volume-asymptotic-grushin}
    \mu_{\gamma, \, p}(A_K) \lesssim K^{m-2\beta-(2\beta-1)n/\alpha}e^{-K^{2\beta}}.
\end{equation}
For the lower bound, returning to \eqref{volume-start}, by \eqref{exponential-asymptotic}
\begin{align}
    \mu_{\gamma, \, p}(A_K) &\gtrsim \int_0^\infty x^{n-1}\frac{(x^\alpha + K)^{m-1}e^{-(x^{2\alpha} + (x^\alpha + K)^2)^\beta}}{(x^\alpha + K)(x^{2\alpha} + (x^\alpha + K)^2)^{\beta-1} - (m-1)\log(x^\alpha + K)}dx \notag \\
    &\gtrsim \int_0^\infty x^{n-1}(x^\alpha + K)^{m-2\beta} e^{-(2x^\alpha + K)^{2\beta}}dx. \vphantom{\int_0^\infty x^{n-1}\frac{(x^\alpha + K)^{m-1}e^{-(x^{2\alpha} + (x^\alpha + K)^2)^\beta}}{(x^\alpha + K)(x^{2\alpha} + (x^\alpha + K)^2)^{\beta-1} - (m-1)\log(x^\alpha + K)}dx} \label{volume-intermediate-2}
\end{align}
Note we used \eqref{exponential-asymptotic} with a function $\phi = \phi_x$ which depends on $x$. It is readily checked that it continues to hold provided $e^{-\varphi_x(y)}/\varphi'_x(y) \rightarrow 0$ as $y \rightarrow \infty$. This brings us back to \eqref{volume-intermediate-1} with $2x^\alpha$ replacing $x^\alpha$ in the exponential and by the same arguments as before we recover the expected asymptotic \eqref{volume-asymptotic-grushin}, the only modification being that for $2\beta$ noninteger, we write $2\beta = 2\beta - \floor{2\beta} + \floor{2\beta}$, expand the integer part, and then apply Bernoulli's inequality (going in the opposite direction) for fractional exponents.

\subsection{Estimates for the perimeter}
For the perimeter, the formula \eqref{perimeter-grushin} together with the fact the (unnormalised) normal is 
\[ \nabla((1 + \gamma)\abs{y} - \abs{x}^{1 + \gamma}) = (-(1 + \gamma)\abs{x}^\gamma \nabla_x\abs{x}, (1 + \gamma)\nabla_y\abs{y}), \]
we have
\begin{align}
    \mu_{\gamma, \, p}^+(A_K) &= \int_{\partial A_K} \sqrt{\frac{\abs{x}^{2\gamma} + \abs{x}^{2\gamma}}{1 + \abs{x}^{2\gamma}}} \exp(-\abs{x}^{2(1+\gamma)} + (1+\gamma)^2\abs{y}^2)^{p/(2(1+\gamma))}d\cH^{n+m-1}(x, y) \notag \\
    &\simeq \int_0^\infty x^{n-1}(x^\alpha + K)^{m-1} \sqrt{\frac{2x^{2(\alpha-1)}}{1 + x^{2(\alpha-1)}}}e^{-(x^{2\alpha} + (x^\alpha + K)^2)^\beta}dx \vphantom{\int_{\partial A_K} \sqrt{\frac{\abs{x}^{2\gamma} + \abs{x}^{2\gamma}}{1 + \abs{x}^{2\gamma}}}e^{-(\abs{x}^{2(1+\gamma)} + (1+\gamma)^2\abs{y}^2)^{p/(2(1+\gamma))}}d\cH^{n+m-1}(x, y)} \notag \\
    &\simeq \int_0^\infty \frac{x^{n+\alpha-2}}{\sqrt{1 + x^{2(\alpha-1)}}}(x^\alpha + K)^{m-1}e^{-(2x^{2\alpha} + 2x^\alpha K + K^2)^\beta}dx. \vphantom{\int_{\partial A_K} \sqrt{\frac{\abs{x}^{2\gamma} + \abs{x}^{2\gamma}}{1 + \abs{x}^{2\gamma}}}e^{-(\abs{x}^{2(1+\gamma)} + (1+\gamma)^2\abs{y}^2)^{p/(2(1+\gamma))}}d\cH^{n+m-1}(x, y)} \label{perimeter-intermediate-1}
\end{align}
Observing $(x^\alpha + K)^2 \leq 2x^{2\alpha} + 2x^\alpha K + K^2 \leq (2x^\alpha + K)^2$ and following previous arguments, in particular since \eqref{perimeter-intermediate-1} has the same form (the square root excluded momentarily) as \eqref{volume-intermediate-1} and \eqref{volume-intermediate-2} but with $n+\alpha-1$ replacing $n$ and $m-1$ replacing $m-2\beta$, we conclude the perimeter of $A_K$ enjoys the asymptotics
\begin{equation}\label{perimeter-asymptotic-grushin} 
    \mu_{\gamma, \, p}^+(A_K) \simeq K^{m-1-(2\beta-1)(n+\alpha-1)/\alpha}e^{-K^{2\beta}} = K^{m-2\beta-(2\beta-1)n/\alpha+(2\beta-1)/\alpha}e^{-K^{2\beta}}. 
\end{equation}
Note the square root factor does not contribute to the leading asymptotics up to constants since after the two changes of variable, we see with $q = 2(\alpha-1)$ that $x^{2(\alpha-1)} = x^q$ is replaced by $K^{-(2\beta-1)q/\alpha}w^{q/\alpha}$. Since the exponent on $K$ is negative, the square root can be bounded for large $K$ above by $\sqrt{1 + w^{q/\alpha}}$ and below by $1$. Either way, it ultimately only contributes a constant to the asymptotics. 

The perimeter thus enjoys an improvement by a factor of $K^{(2\beta-1)/\alpha}$ over the volume. This is exactly the improvement predicted by the asymptotic \eqref{isoperimetric-model-asymptotics}; with the notation of this section, the isoperimetric inequality \eqref{isoperimetric-grushin} holds with $r = 2\alpha\beta/(1 + 2(\alpha-1)\beta)$ and so 
\[ \log(1/\mu_{\gamma, \, p}(A_K))^{1 - 1/r} \gtrsim K^{2\beta(1-1/r)} = K^{(2\beta - 1)/\alpha} \]
which gives
\[ \frac{\mu^+_{\gamma, \, p}(A_K)}{\mu_{\gamma, \, p}(A_K)\log(1/\mu_{\gamma, \, p}(A_K))^{1 - 1/r}} \lesssim 1 \]
as $K \rightarrow \infty$. 

\begin{remark}\label{remark-1}
    The normalisation $(1 + \gamma)\abs{y}$ defining $A_K$ is a matter of convenience and simplifying some computations; in particular the constant $1 + \gamma$ can be replaced by any other positive constant and Theorem \ref{thm-grushin} continues to hold. Moreover, the dependence of the estimates on the dimensions $n$ and $m$ can be removed. For instance, if $\gamma \in \bbZ_{\geq 1}$ and $\beta = 1$ then one can show, because $\mu_{\gamma, \, p} = \mu_{\gamma, \, 2(1 + \gamma)}$ is a gaussian in $y$, sets of the form 
    \[ \{\inner{v, y} > x_1^{1 + \gamma} + K\} \]
    for some $v \in \bbS^{m - 1}$ are almost isoperimetric extremisers enjoying volume and perimeter estimates independent of $n$ and $m$. In the general case $\gamma \geq 0$ and $\beta \geq \frac{1}{2}$ one can show sets of the form $\{y_1 > \abs{x_1}^{1 + \gamma} + K\}$ are almost extremisers satisfying dimension free estimates.
\end{remark}

%% file: heisenberg.tex
\section{The step two setting}
As mentioned in the introduction, in the setting of the $3$-dimensional Heisenberg group $\bbH^1$, and more generally of a step two stratified Lie group $\bbG \cong \bbR^n_x \times \bbR^m_z$, we have no analogue of the Grushin perimeter formula \eqref{perimeter-grushin}. For the volume asymptotics however, note the proof in the previous chapter did not depend on the Grushin metric, and so we can recycle entirely the computations for $\gamma = 1$. Thus if we take $\alpha = 2$ in the Grushin volume asymptotics \eqref{volume-asymptotic-grushin}, we arrive at 
\begin{equation}\label{volume-asymptotic-group}
    \mu_{\bbG, \, p}(A_K) = Z^{-1}\int_{A_K} \exp(-(\abs{x}^4 + \abs{z}^2))^{p/4}dzdx \simeq K^{m-2\beta-n(2\beta-1)/2}e^{-K^{2\beta}} 
\end{equation}
for $p = 4\beta$ with $\beta \geq \frac{1}{2}$. 

For the perimeter asymptotics, which as it turns out are again the Grushin ones \eqref{perimeter-asymptotic-grushin}, we return to the original definition of perimeter \eqref{def-perimeter} and give explicit estimates for the $\epsilon$-enlargement $A_{K, \, \epsilon}$ of $A_K$. While there is no explicit formula for the distance between two points on the Heisenberg group (save for some special cases) let alone in general, what is known is that $d_{\bbG}(x, 0)$ is a homogeneous norm in the sense of \cite[\S5.1]{bonfiglioli2007stratified} with respect to the family of anisotropic dilations on $\bbG$ defined by $\delta_\lambda(x, z) \mapsto (\lambda x, \lambda^2z)$, $\lambda > 0$, and thus, since all homogeneous norms are mutually equivalent by \cite[Proposition~5.1.4]{bonfiglioli2007stratified}, it can be estimated in particular by the Kaplan norm \eqref{def-kaplan-norm}.

To provide some intuition for the structure of $A_{K, \, \epsilon}$, we first provide a sketch of the heuristics in the setting of the $3$-dimensional Heisenberg group $\bbH^1 \cong \bbR^2_{x, y} \times \bbR^1_z = (\bbR^1_x \times \bbR^1_y) \times \bbR^1_z$ equipped with the group law 
\[ (x_1, y_1, z_1) \circ (x_2, y_2, z_2) = \left(x_1 + y_1, x_2 + y_2, z_1 + z_2 + 2(x_1y_2 - x_2y_1)\right). \]
By \cite[Proposition~5.2.4]{bonfiglioli2007stratified}, the distance $d_{\bbH^1}(g, h)$ between two points $g, h \in \bbH^1$ is the distance between $g \circ h^{-1}$ and the identity element $e = 0$ where inversion on $\bbH^1$ is euclidean inversion $h \mapsto h^{-1} = -h$. Since $x \mapsto d_{\bbH^1}(x, 0)$ is a homogeneous norm by \cite[Theorem~5.2.8]{bonfiglioli2007stratified}, its aforementioned equivalence with the Kaplan norm implies 
\begin{align}
    d_{\bbH^1}((u, v, w), (x, y, z))^4 
    &= d_{\bbH^1}((u-x, v-y, w-z + 2(vx - uy)), 0)^4. \notag \\ 
    &\lesssim ((u - x)^2 + (v - y)^2)^2 + (w - z + 2(vx - uy))^2. \label{comparison}
\end{align}
If $d_{\bbH^1}((u, v, w), (x, y, z))$ is comparable to $\epsilon > 0$ and we write $u = x + \delta_1$, $v = y + \delta_2$, and $w = z + \delta_3$ for some $\delta_1, \delta_2, \delta_3 \in \bbR$, we see that \eqref{comparison} implies
\begin{equation}\label{set-delta-conditions}
    \abs{\delta_1}, \abs{\delta_2} \leq C_0\epsilon, \quad \abs{\delta_3 + 2(\delta_2x - \delta_1y)} \leq C_0\epsilon^2
\end{equation}
for some $C_0 > 0$. It turns out that the salient estimate is the former in the sense although the latter also imposes conditions on $\delta_1, \delta_2$, the worst case scenario of a point $(x, y, z)$ within $\epsilon$-distance to $(u, v, w) \in A_K$ is achieved when $\abs{\delta_1}, \abs{\delta_2} = C_0\epsilon$ and heuristically we expect if $\abs{z} > x^2 + y^2 + K$ then  
\[ \abs{w} > u^2 + v^2 - C\epsilon(\abs{u} + \abs{v}) + K + \text{something of lower order} \] 
for some $C > 0$. While this implies an inclusion for $A_{K, \, \epsilon}$ in the direction we need (since we need an upper bound on the perimeter which amounts to finding a small enough set containing $A_{K, \, \epsilon}$), we prove also an inclusion in the opposite direction for completeness. 

\begin{proposition}\label{prop1}
    For all $\epsilon > 0$ sufficiently small, there exist $C_1, C_2 > 0$ such that 
    \[ A_{K, \, \epsilon} \sbs \{\abs{z} > x^2 + y^2 - C_1\epsilon (\abs{x} + \abs{y}) + K - C_2\epsilon^2\}. \]
    and there exist $C_3 > 0$ such that 
    \[ \{x^2 + y^2 \leq 1\} \cap \{\abs{z} > x^2 + y^2 - C_3\epsilon(\abs{x} + \abs{y}) + K\} \sbs A_{K, \, \epsilon}. \]
\end{proposition}

\begin{proof}
    For the former inclusion, let $(x, y, z) \in A_{K, \, \epsilon}$, meaning there exists $(u, v, w) \in A_K$ satisfying $\abs{w} > u^2 + v^2 + K$ such that $d((u, v, w), (x, y, z)) < \epsilon$. We wish to show $\abs{z} > x^2 + y^2 - C_3\epsilon(\abs{x} + \abs{y}) + K - C_4\epsilon^2$. Assume for now $w > 0$. Let 
    \[ (u, v, w) = (x + \delta_1, y + \delta_2, z + \delta_3). \]
    Starting from \eqref{comparison}, we have
    \begin{align*}
        \abs{z} > z &> w + 2(vx - uy) - C_0\epsilon^2 \\
        &> u^2 + v^2 + 2(vx - uy) + K - C_0\epsilon^2 \\ 
        &= (x + \delta_1)^2 + (y + \delta_2)^2 + 2((y + \delta_2)x - (x + \delta_1)y) + K - C_0\epsilon^2 \\ 
        &= x^2 + y^2 + 2(\delta_1 + \delta_2)x + 2(-\delta_1 + \delta_2)y + K - C_0\epsilon^2 \\ 
        &\geq x^2 + y^2 - 4C_0\epsilon(\abs{x} + \abs{y}) + K - C_0\epsilon^2
    \end{align*}
    since \eqref{set-delta-conditions} implies $\abs{\delta_1}, \abs{\delta_2} \leq C_0\epsilon$. Thus the proposition holds with $C_1 = 4C_0$ and $C_2 = C_0$. The case of $w < 0$ is similar.  

    For the latter inclusion, let $(x, y, z)$ satisfy 
    \[ \abs{z} > x^2 + y^2 - C_3\epsilon(\abs{x} + \abs{y}) + K \]
    and $x^2 + y^2 \leq 1$. We wish to show there exists $(u, v, w) \in A_K$ satisfying $\abs{w} > u^2 + v^2 + K$ such that $d((u, v, w), (x, y, z)) < \epsilon$. Assume for now $z > 0$ and $0 < \epsilon \leq 1$. Let 
    \[ (u, v, w) = (x + \delta_1, y + \delta_2, z + \delta_3) \]
    with $\delta_1 = -\sign(x)C_0\chi_1\epsilon$, $\delta_2 = -\sign(y)C_0\chi_2\epsilon$, and $\delta_3 = C_0\epsilon^2 - 2(\delta_2x - \delta_1y)$, for some $0 < \chi_1, \chi_2 < 1$ to be determined. 

    Since $x^2 + y^2 \leq 1$ and $0 < \epsilon \leq 1$ implies that $z > x^2 + y^2 - C_3\epsilon(\abs{x} + \abs{y}) + K > K - 2C_3$ and $\abs{\delta_3} \leq C_0\epsilon^2 + \abs{2(\delta_2x - \delta_1y)} \leq 5C_0$, for $K$ sufficiently large it holds $z + \delta_3 > 0$ and thus 
     \begin{align*}
        \abs{w} &= z + \delta_3 \\
        &> x^2 + y^2 - 2(\delta_2x - \delta_1y) - C_3\epsilon(\abs{x} + \abs{y}) + K + C_0\epsilon^2 \\
        &= (u - \delta_1)^2 + (v - \delta_2)^2 - 2(\delta_2x - \delta_1y) - C_3\epsilon(\abs{x} + \abs{y}) + K + C_0\epsilon^2 \\
        &= u^2 + v^2 + 2(-\delta_1-\delta_2)x + 2(\delta_1 - \delta_2)y - \delta_1^2 - \delta_2^2 - C_3\epsilon(\abs{x} + \abs{y}) + K + C_0\epsilon^2
    \end{align*}
    Firstly, we choose $\chi_1, \chi_2$ such that $\delta_1^2 + \delta_2^2 = C_0^2(\chi_1^2 + \chi_2^2) < C_0$ and hence the constant term $C_0\epsilon^2 - \delta_1^2 - \delta_2^2$ is nonnegative, so for instance we may take $\chi_1, \chi_2 \leq \frac{1}{2}\sqrt{C_0}$. Secondly, we note one and exactly one of $-\delta_2x$ or $\delta_1y$ is positive. Supposing for the moment $-\delta_2x > 0$, in which case $\delta_1y = -C_0\chi_1\epsilon\abs{y} < 0$, then choosing $\chi_1 = \chi_2/2$ and $\chi_2 = \frac{1}{2}\sqrt{C_0}$, we obtain 
    \[ 2(\delta_1 - \delta_2)y = C_0\chi_2\epsilon\abs{y} \Mand 2(-\delta_1 - \delta_2)x = 3C_0\chi_2\epsilon\abs{x}. \]
    Thus the proposition holds with $C_3 = \frac{1}{2}C_0\chi_2 = \frac{1}{4}C_0^{3/2}$. The case of $-\delta_2x < 0$ is proved by interchanging $\chi_1$ and $\chi_2$. 
\end{proof}

We now prove the perimeter asymptotics. By the previous proposition, the difference in volume between $A_K$ and its $\epsilon$-enlargement $A_{K, \, \epsilon}$ is bounded below by 
\begin{align*}
    \mu_{\bbH^1, \, p}(A_{K, \, \epsilon} \setminus A) &\gtrsim \int_{0}^1 \int_{r^2 - C_1\epsilon r + K}^{r^2 + K} r^{2-1}z^{1-1}e^{-(r^4 + z^2)^\beta}dzdr \gtrsim \epsilon \int_0^1 r^{3-1}e^{-(r^4 + (r^2 + K)^2)^\beta}dr
\end{align*}
after passing to radial coordinates and since $\abs{x} + \abs{y}$ is comparable to $r = \sqrt{x^2 + y^2}$. Compared to \eqref{volume-intermediate-1}, this is almost the same integral in form except we have a finite region of integration in the $r$-variable. Since the changes of variable lead to $w = K^{2\beta}z = K^{2\beta-1}r^\alpha$, the upper terminal transforms into $K^{2\beta-1}$ which for $K > 1$ is bounded below by $1$, in particular away from zero, meaning we have the same asymptotics but with a different finite integral. Thus after dividing by $\epsilon$ and sending $\epsilon \rightarrow 0^+$, we obtain the lower bound 
\[ \mu_{\bbH^1, \, p}^+(A_K) \gtrsim K^{-3(2\beta-1)/2}e^{-K^{2\beta}} \] 
which is consistent with the Grushin asymptotics as claimed earlier by taking $n = 2$, $m = 1$, and $\alpha = 2$ in \eqref{perimeter-asymptotic-grushin}, and which improves the volume asymptotics \eqref{volume-asymptotic-group} by the expected factor $K^{(2\beta-1)/\alpha} = K^{(2\beta-1)/2}$. Similarly, $\mu_{\bbH^1, \, p}(A_{K, \, \epsilon} \setminus A)$ is bounded above by 
\begin{align*}
    \mu_{\bbH^1, \, p}(A_{K, \, \epsilon} \setminus A) &\lesssim \int_0^\infty \int_{r^2 - C_1\epsilon r + K - C_2\epsilon^2}^{r^2 + K} r^{2-1}z^{1-1}e^{-(r^4 + z^2)^\beta}dzdr \\
    &\lesssim \int_0^\infty (\epsilon r + \epsilon^2)r^{2-1}e^{-(r^4 + (r^2 - C_1\epsilon r + K - C_2\epsilon)^2)^\beta}dr.
\end{align*}
The $\epsilon^2$ term is of lower order while the $\epsilon r$ term contributes the same integral as before but taken over the entirety of $r \in \bbR_{\geq 0}$ and thus enjoys the same asymptotics. 

This completes the proof of the volume and perimeter asymptotics of $A_K$ in the setting of the $3$-dimensional Heisenberg group. To obtain the general result for an arbitrary step two stratified Lie group $\bbG \cong \bbR^n_x \times \bbR^m_z$, let us recall that the group law \eqref{def-step-two-law} as characterised by \cite[Theorem~3.2.2]{bonfiglioli2007stratified} takes on the form 
\[ (x, z) \circ (\xi, \zeta) = \left(x + \xi, z_1 + \zeta_1  + \frac{1}{2}\inner{B^{(1)}x, \xi}, \cdots, z_m + \zeta_m + \frac{1}{2}\inner{B^{(m)}x, \xi}\right) \]
for a collection of $m$ linearly independent skew-symmetric matrices $B^{(1)}, \cdots, B^{(m)}$ of dimension $n \times n$. The point is skew-symmetry of the matrices implies the worst case scenario loss in the $\epsilon$-enlargement happens once again in only the first order terms. In particular, the Heisenberg group law $2(vx - uy)$ in the previous proof is replaced by $\frac{1}{2}\inner{B^{(j)}x, \xi}$ and if $\xi = x + \delta$ for some $\delta \in \bbR^n$ such that $\norm{\delta}_\infty \lesssim \epsilon$, then
\[ \smallabs{\inner{B^{(j)}x, \xi}} = \smallabs{\inner{B^{(j)}x, \delta}} \lesssim \norm{\delta}_\infty \abs{x} \lesssim \epsilon \abs{x}. \]
The remainder of the argument goes in exactly the same way with $d_\bbG$ compared to $N_\bbG$ and the integral 
\[ \liminf_{\epsilon \to 0^+} \frac{1}{\epsilon} \int_0^\infty \int_{r^2-c_1\epsilon r + K - c_2\epsilon^2}^{r^2 + K} r^{n-1}z^{m-1}e^{-(r^4 + z^2)^\beta}dzdr \]
gives the expected asymptotics, that is improves \eqref{volume-asymptotic-group} by the factor $K^{(2\beta-1)/\alpha}$. 

\begin{remark}
    This argument only generalises the first inclusion of Proposition \ref{prop1}. The second inclusion is more involved and requires inductively defining the $\chi_i$.
\end{remark}
    
The same observations in Remark \ref{remark-1} can be made again here. In light of the structure of the almost extremisers, it is somewhat tempting to conjecture such sets are actually extremisers achieving the sharp isoperimetric inequality at least in the smooth setting of $\beta = 1$, for instance sets of the form $\{y_1 > x_1^{1+\gamma} + K\}$ on the Grushin space with $\gamma \in \bbZ_{\geq 1}$, and possibly assuming $n = 1$, or otherwise sets of the form $\{z_1 > x_1^2 + K\}$ on a step two group, and possibly isomorphic to $\bbH^1$.